\documentclass[leqno,12pt]{article}
\usepackage[margin=1.3in, top = 1.3in, bottom=1.3in]{geometry}
\usepackage[utf8]{inputenc}
\usepackage{amsmath}
\usepackage{amsfonts}
\usepackage{mathrsfs}
\usepackage{tikz}
\usepackage{hyperref}
\usepackage{stmaryrd}
\usepackage{enumerate}
\usepackage{tikz-qtree}
\usetikzlibrary{automata,positioning}
\usepackage{amssymb}
\usepackage{amsthm}
\usepackage{bbm}
\usepackage{mathtools}
\usepackage{algorithm}
\usepackage{algpseudocode}
\usepackage{diffcoeff}
\usepackage{interval}
\usepackage{graphicx}
\usepackage{comment}
\usepackage{fancyhdr}
\usepackage{tikz}
\usepackage{tikz-3dplot}
\usepackage{enumitem}
\usetikzlibrary{shapes,arrows}
\theoremstyle{plain}
\newcommand{\R}{\mathbb{R}}
\newcommand{\st}[1]{\left\{#1\right\}}

\newcommand{\from}{\colon}

\DeclareMathOperator{\prox}{prox}

\newtheorem{theorem}{Theorem}[]
\theoremstyle{definition}
\newtheorem{definition}{Definition}[]

\theoremstyle{plain}
\newtheorem{lemma}[]{Lemma}

\newtheorem{innercustomthm}{Assumption}
\newenvironment{customthm}[1]
  {\renewcommand\theinnercustomthm{#1}\innercustomthm}
  {\endinnercustomthm}

\tikzstyle{decision} = [diamond, draw, fill=blue!20, 
text width=6em, text badly centered, node distance=3cm, inner sep=0pt]
\tikzstyle{block} = [rectangle, draw, fill=blue!20, 
text width=9em, text centered, rounded corners, minimum height=4em]
\tikzstyle{block2} = [rectangle, draw, fill=green!20, 
text width=9em, text centered, rounded corners, minimum height=4em]
\tikzstyle{line} = [draw, -latex']
\tikzstyle{cloud} = [draw, ellipse,fill=red!20, node distance=3cm,
minimum height=4em]

\begin{document}
\title{\vspace{-2.5cm} Minimal enclosing balls via geodesics}
\author{Ariel Goodwin
	\thanks{CAM, Cornell University, Ithaca, NY, USA.
	\texttt{awg77@cornell.edu}  Research supported in part by the NSERC Postgraduate Fellowship PGSD-587671-2024.}
\and
Adrian S. Lewis
\thanks{ORIE, Cornell University, Ithaca, NY, USA.
\texttt{adrian.lewis@cornell.edu} Corresponding Author.
Research supported in part by National Science Foundation Grant DMS-2405685.}
}

\date{\today}

	\maketitle

\begin{abstract}
Algorithms for minimal enclosing ball problems are often geometric in nature.  To highlight the metric ingredients underlying their efficiency, we focus here on a particularly simple geodesic-based method.  A recent subgradient-based study proved a complexity result for this method in the broad setting of geodesic spaces of nonpositive curvature. We present a simpler, intuitive and self-contained complexity analysis in that setting, which also improves the convergence rate.  We furthermore derive the first complexity result for the algorithm on geodesic spaces with curvature bounded above.

\end{abstract}

\medskip

\medskip

The \textit{minimal enclosing ball problem} for a bounded set $A$ in a metric space $(X,d)$ seeks a ball of minimal radius containing $A$.  Specifically, we seek a {\em center\/}: a point $x \in X$ minimizing the objective value $f(x) = \sup_A d(x,\cdot)$. A relaxed version instead seeks an {\em approximate center\/}:  a point $x\in X$ whose optimality gap $f(x) - \inf f$ is bounded by some given tolerance.

	Invitingly simple in metric formulation, minimal enclosing balls
	have widespread application throughout computer science and operations research.
	For finite subsets of the plane, Sylvester studied the problem in 1857 \cite{Sylvester1857}, and  the intervening 150 years have seen a wide range of sophisticated algorithms in Euclidean space. 
	
Non-Euclidean versions of the problem are also familiar.  Operations researchers, for example, study \textit{minimax facility location}, placing a facility to minimize worst-case distance to several given locations in a network.  Modeling the network as a \textit{metric tree} $X$ --- an undirected tree graph with nonnegative edge lengths and distance $d$ measuring the path length between any two points in $X$ --- produces a minimal enclosing ball problem.  The following 1973 algorithm of Handler \cite{Handler1973} is strikingly simple.
\begin{quote}
From any initial point in $A$, find a furthest point $a$ in $A$.  From $a$, find a furthest point $b$ in $A$.  The midpoint of $a$ and $b$ is the unique center of~$A$.
\end{quote}

	While not quite rivaling the simplicity of Handler's method, which needs just $2|A|-2$ distance computations and one midpoint calculation, a common algorithmic strategy for the minimal enclosing ball problem also relies only on distance and path computations:  at each iteration we simply choose a point in $A$ furthest from the current iterate and move some distance towards it. This strategy is well-defined in particular in any \textit{geodesic metric space} $(X,d)$.  Given any points $x$ and $y$ in such a space, at a distance $d(x,y)=\delta$, there is an isometry $\gamma\from [0, \delta]\to X$ 
	with $\gamma(0) = x$ and $\gamma(\delta) = y$. We call $\gamma$ a \textit{geodesic} joining $x$ to $y$.  When such geodesics are always unique, we call the space \textit{uniquely geodesic}.  In that case, 
	we denote the corresponding geodesic image by $[x,y]$, and for each $t\in [0,1]$ we denote by $(1-t)x+ty$ the unique point on $[x,y]$ at a distance $t\delta$ from $x$. With this notation, we study the following procedure.
	
	\begin{algorithm}[h!]
		\caption{Approximate Minimal Enclosing Ball}
			\begin{algorithmic}
				\Require $\st{t_k}_{k=0}^\infty \subset [0,1]$ 
				\State Choose $x^0 \in A$
				\For{$k=0,1,2,\dots$}
				\State From $x^k$, choose a furthest point $\bar a^k$ in $A$
				\State $x^{k+1} \leftarrow (1-t_k)x^k + t_k \bar a^k$
				\EndFor
				\end{algorithmic}
	\label{GeodesicAlgorithm}
	\end{algorithm}

\noindent
Various assumptions on the set $A$, such as finiteness or compactness, can ensure the existence of the requisite furthest points.
	The method's performance depends on the choice of stepsize sequence $\{t_k\}$:  typical predetermined schemes involve constant or slowly decaying sequences like $\{\frac{1}{k}\}$ or 
	$\{\frac{1}{\sqrt{k}}\}$.  
	 
	 In Euclidean space, Algorithm \ref{GeodesicAlgorithm} was introduced by \cite{Badoiu2003}, and then generalized for minimal enclosing balls with respect to Bregman divergences in \cite{Nock2005}, and on Riemannian manifolds of bounded sectional curvature in \cite{Arnaudon2013}.  Recently, \cite{LewisHoroballs2024} studied the complexity of Algorithm \ref{GeodesicAlgorithm} 
	in the broad setting of \textit{Hadamard spaces}: complete 
	geodesic metric spaces of nonpositive curvature. This curvature condition, known as the \textit{CAT(0) property}, amounts to the following simple strong convexity property.
	\begin{definition}
		\label{def:CAT(0)}
		A geodesic metric space $(X,d)$ is CAT(0) if for any point $z\in X$ and any geodesic $\gamma$ the map
		$t\mapsto d^2(\gamma(t),z)-t^2$ is convex. 
\end{definition}
CAT(0) spaces are uniquely geodesic \cite[Lemma 1.2.2]{Bacak2014}. Then for any points $x,y,z\in X$, the convexity property in Definition \ref{def:CAT(0)} for the geodesic joining $x$ to $y$ amounts to the following inequality: for any scalars $\alpha,\beta\in [0,1]$ such that
$\alpha + \beta =1$,
\begin{equation}
		\label{eqn:CAT(0)ineq}
		d^2(\alpha x+ \beta y,z) ~\leq~ \alpha d^2(x,z) + \beta d^2(y,z) -  \alpha\beta d^2(x,y).
\end{equation}

The Hadamard space setting covers in particular Hilbert space and all complete simply-connected Riemannian manifolds of nonpositive sectional curvature, such as hyperbolic space, and
	positive-definite matrix space with its affine-invariant metric \cite{Boumal2023}.  However, Hadamard spaces need not be linear or manifolds.  For example, metric trees and, more generally, CAT(0) cubical complexes are Hadamard spaces, with distinctive features like bifurcating geodesics \cite{Ardila2012}.
	The \textit{BHV space} 
	of phylogenetic trees is a CAT(0) cubical complex with interesting biological applications
	\cite{Billera2001}.  In these examples of Hadamard spaces, geodesics are computationally tractable \cite{Hayashi2021,Owen2011}, making Algorithm \ref{GeodesicAlgorithm} an inviting tool.
	
The simplicity of Algorithm \ref{GeodesicAlgorithm} stands in stark contrast to published complexity analyses.  In particular, \cite{LewisHoroballs2024} specializes a general subgradient-type algorithm
on Hadamard spaces, and the Riemannian analysis of \cite{Arnaudon2013} relies on both smoothness and curvature bounds.  In fact, however, Algorithm \ref{GeodesicAlgorithm} admits a strikingly elementary complexity analysis in Hadamard space, relying only on Definition \ref{def:CAT(0)}. The argument we present, like the algorithm itself, uses only geodesics:  we make no use of tools from metric geometry like (sub)gradients, tangent cones, exponential maps, rays, or angles. 

The following assumption describes our basic setting and notation.
	\begin{customthm}{A}~
		\label{assumptionA}
		\begin{enumerate}[label=(\roman*)]
			\item $(X,d)$ is a Hadamard space. 
			\item $A$ is a bounded subset of $X$, with diameter $D$.
			\item Corresponding to each point in $X$ is at least one furthest point in $A$.
			\item The objective function $F \colon X \to \R$ is defined by $F(x) = \max_{a\in A}d^2(x,a)$.
		\end{enumerate}
	\end{customthm}

\noindent
The center $x_A$ of $A$ is the unique minimizer of $F$ \cite[II.2.7]{Bridson1999}.  Notice that the furthest point property in condition (iii) holds in particular if $A$ is nonempty and compact. 

A standard computational approach for convex optimization on Hadamard spaces --- indeed the only known general methodology --- is the {\em proximal point method} surveyed in \cite{Bacak2014}.  This method is easy to implement for computing means and medians of finite sets $A$, resulting in algorithms reminiscent of Algorithm \ref{GeodesicAlgorithm} but with the point $\bar a^k$ chosen cyclically or randomly.  The proximal methodology even extends beyond Hadamard spaces, to more general geodesic spaces \cite{Ohta2015}.  Unfortunately, unlike mean and median problems, the minimal enclosing ball problem involves proximal point subproblems of the form
\[
\min_{x \in X} \max_{a \in A} \Big\{ d^2(x,a) + \frac{1}{2\lambda} d^2(x,x^k) \Big\}.
\]
Since this problem seems no easier than the minimal enclosing ball problem itself, we focus instead on the simple geometric approach of Algorithm \ref{GeodesicAlgorithm}.

Algorithm \ref{GeodesicAlgorithm} relies on access to furthest points in the set $A$, which is easy if $A$ is finite. A straightforward extension arises when $A$ is a finite union of closed metric balls $B(a_i, r_i)$, for $i=1,2,\ldots,m$.   Suppose that $X$ has the \textit{geodesic extension property}, meaning that every geodesic in $X$ extends to an isometry with domain
	$\R_+$. To find a furthest point from $x$ in $A$, extend each geodesic from $x$ to the center $a_i$
a distance $r_i$, and then from among the resulting points, choose one furthest from~$x$. 

Curvature conditions like (\ref{eqn:CAT(0)ineq}) naturally lead to complexity results in terms of the squared-distance objective $F$, rather than the objective $f = \sqrt{F}$ that we introduced at the outset.  As we note briefly after the main results, this shift in objective changes the constant in the complexity estimate, but not the {\em rate} of convergence.

	\begin{theorem}
	\label{thm:FastComplexity}
	If Assumption \ref{assumptionA} holds, then any initial sequence of points $x^0,x^1,\ldots,x^N$ generated by Algorithm \ref{GeodesicAlgorithm} with stepsizes
	$t_k = \frac{2}{k+1}$ satisfies
		\[\min_{k=0,\dots,N}F(x^k) - \min F ~\leq~ \frac{4D^2}{N+1}.\]
		\end{theorem}

	\begin{proof}
		For each iteration $k=0,1,\ldots,N-1$, inequality \eqref{eqn:CAT(0)ineq} implies
		\begin{align}
			d^2(x^{k+1},x_A) &= d^2( (1-t_k)x^k + t_k\bar a^k,x_A) \nonumber \\
			&\leq (1-t_k)d^2(x^k,x_A) + t_k d^2(\bar a^k,x_A) -t_k(1-t_k)d^2(x^k,\bar a^k) \nonumber \\
			&\leq (1-t_k)d^2(x^k,x_A) - t_k(F(x^k) - \min F) + t_k^2 F(x^k), \label{first}
	\end{align}
	where the last inequality uses the property $d^2(\bar a^k,x_A) \leq F(x_A) = \min F$ and the 
	identity $F(x^k) = d^2(x^k,\bar a^k)$ is a consequence of the definition of $\bar a^k$.
	Note that $F(x^0) \leq D^2$ because $x^0 \in A$. Then, assuming $F(x^k) \leq D^2$, inequality \eqref{eqn:CAT(0)ineq}
	leads to	
	\[F(x^{k+1}) \leq (1-t_k)F(x^k) + t_k F(\bar a^k) \leq (1-t_k)D^2 + t_k D^2 = D^2.\]
	Thus $F(x^k) \leq D^2$ for all $k$ by induction, from which inequality \eqref{first} implies  
	\begin{equation}
		\label{eqn:termk2}
		t_k(F(x^k) - \min F) \leq (1-t_k)d^2(x^k,x_A) - d^2(x^{k+1},x_A) +t_k^2 D^2.
	\end{equation}
	Divide by $t_k$ and substitute $t_k = \frac{2}{k+1}$ to obtain 
	\[
		F(x^k) - \min F\leq \frac{(k-1)}{2} d^2(x^k,x_A) - \frac{(k+1)}{2}d^2(x^{k+1},x_A) + \frac{2D^2}{k+1}.
	\]
	Multiply by $k$ and sum over $k=0,\dots,N$ so the squared distance terms telescope:
	\[
		\sum_{k=0}^Nk(F(x^k)-\min F) ~\leq~ 0 -\frac{N(N+1)}{2}d^2(x^{N+1},x_A) + 2D^2\sum_{k=0}^N\frac{k}{k+1} \leq 2D^2N.
\]
Since $F(x_k) \ge \min\{F(x_j) : 0 \le j \le N \}$ for all $k$, the result now follows.
	\end{proof}

\noindent
Theorem \ref{thm:FastComplexity} improves the $O(\frac{1}{\sqrt{N}})$ convergence rate proved in \cite{LewisHoroballs2024} to $O(\frac{1}{N})$, a rate previously known only for the special case of Hadamard manifolds \cite{Criscitiello2025}.
	 
Beyond the uniquely geodesic property, the proof of Theorem~\ref{thm:FastComplexity} relies on CAT(0) geometry only to obtain inequality \eqref{eqn:termk2}, after which the argument follows a standard Euclidean model \cite[Theorem 8.31]{Beck2017}.  We therefore next explore relaxing Assumption \ref{assumptionA} to the class of 
CAT($\kappa$) spaces with curvature $\kappa > 0$.  Such spaces are exemplified by Euclidean spheres of radius $\frac{1}{\sqrt{\kappa}}$ with their angular metrics.  	Rather than defining the CAT($\kappa$) property formally, we refer the reader to a comprehensive resource on metric geometry \cite[Chapter II.1]{Bridson1999}.  Worth highlighting, however, is that all Riemannian manifolds with sectional curvature bounded above by $\kappa \ge 0$ are, locally, CAT($\kappa$) spaces:  centered at any point, some ball is a CAT($\kappa$) space \cite[Definition II.1.2 and Theorem II.1A.6]{Bridson1999}.  

Optimal transport has given rise to some recent interest in minimal enclosing balls in Wasserstein spaces \cite{Wang2025}.  While Wasserstein spaces are not in general CAT($\kappa$), we note one interesting 
special case:  the manifold of positive-definite matrices, considered as mean-zero Gaussian measures and equipped with the 2-Wasserstein distance. Restricting attention to those matrices whose eigenvalues are at least $\lambda > 0$ results in a complete Riemannian manifold with sectional curvature at most $3/(2\lambda^2)$ in which geodesics can be computed explicitly:  see \cite[Proposition 2]{Massart2019}, along with \cite{Takatsu2011} and \cite[Theorem 1 and Remarks 2 and 3]{altschuler-chewi-gerber-stromme}.

Convexity-based techniques in CAT($\kappa$) spaces with curvature $\kappa > 0$ depend on restricting the diameter of the space:  otherwise, geodesics may not be unique, balls may not be convex, and minimal enclosing ball problems may have spurious local minima.  We therefore revise Assumption \ref{assumptionA} accordingly.

	\begin{customthm}{B}~
		\label{assumptionB}
		Assumption \ref{assumptionA} holds, modified as follows:
		\begin{enumerate}[label=(\roman*)]
			\item $(X,d)$ is a complete CAT($\kappa$) space, with diameter less than 
				$\frac{\pi}{2\sqrt{\kappa}}$ if $\kappa > 0$.
		\end{enumerate}
	\end{customthm}
	
\noindent
Under Assumption \ref{assumptionB}, the space $X$ is still uniquely geodesic, and the center $x_A$ of $A$ is still the unique minimizer of the objective function $F$ \cite[II.1.4(1) and II.2.7]{Bridson1999}.  Furthermore, for all points $x\in X$ the 
squared distance function
$d^2(\cdot,x)$ is convex when composed with any geodesic (functions with this property are called \textit{convex}), 
and the following relaxed version of inequality \eqref{eqn:CAT(0)ineq} holds.

	\begin{lemma}
		\label{ProxDescent}
		If Assumption \ref{assumptionB}(i) holds, then
		for any points $x,y,z\in X$ and scalars $\alpha,\beta\in (0,1)$ such that $\alpha+\beta =1$, the following inequality holds:
\[	
d^2(\alpha x+ \beta y,z) ~\leq~  
\frac{\alpha d^2(x,z) + \beta d^2(y,z)}{ \max\{\alpha,\beta\} }
~-~ \alpha\beta d^2(x,y).
\]
	\end{lemma}
	
\begin{proof}
Define convex functions $g,h \colon X \to \R$ by $g(q) = \alpha d^2(q,x) + \beta d^2(q,y)$ and $h(q)= d^2(q,x)$ for points $q \in X$. Elementary inequalities show  
\begin{eqnarray*}
g(\alpha x+\beta y) 
&=& 
\alpha\beta d^2(x,y) ~\le~ \alpha\beta \big(d(q,x) + d(q,y)\big)^2 \\
&\le& 
\alpha\beta \big(d^2(q,x) + d^2(q,y)\big) ~+~ \alpha^2 d^2(q,x) ~+~ \beta^2 d^2(q,y) ~=~ g(q).
\end{eqnarray*}
Thus the point $\alpha x + \beta y$ minimizes the function $g$, or equivalently, the function
$h + \frac{\beta}{\alpha}d^2(\cdot,y)$.  Setting $\lambda = \frac{\alpha}{2\beta}$ and using the language of proximal operators \cite{Ohta2015}, we have $\prox_{\lambda h}(y) = \alpha x + \beta y$.  By \cite[Lemma 4.6(I)]{Ohta2015}, for any point $z\in X$ we have
\[
d^2(\prox_{\lambda h}(y),z) ~\leq~ d^2(y,z) - 2\lambda \Big(h\big(\prox_{\lambda h}(y)\big) - h(z)\Big),
\]
or equivalently,
\[
	d^2(\alpha x + \beta y,z) ~\le~ \frac{\alpha}{\beta} d^2(x,z) + d^2(y,z)  - \alpha\beta d^2(x,y).
\]
The result now follows by combining this inequality with the corresponding version obtained by 
swapping the roles of $\alpha$ and $\beta$, and of $x$ and $y$.
	\end{proof}
	
We can now present the first known complexity result for a geodesic-based minimal enclosing ball algorithm in general CAT($\kappa$) spaces.  Unlike Theorem \ref{thm:FastComplexity}, this analysis uses a constant stepsize, associated with the number of iterations $N$, and resulting in a rate of $O(\frac{1}{\sqrt{N}})$
	(compared with the $O(\frac{1}{N})$ rate of Theorem \ref{thm:FastComplexity}). In the special case of Riemannian manifolds, a complexity analysis for a geodesic-based minimal enclosing ball algorithm appears in \cite{Arnaudon2013}, with a guaranteed rate depending on upper and lower curvature bounds \cite[Proposition 1]{Arnaudon2013}.  Notably, we prove a rate that depends on curvature only via the diameter constraint in Assumption \ref{assumptionB}(i).
 
	\begin{theorem}
		\label{thm:cat(k)complexity}
If Assumption \ref{assumptionB} holds, then any initial sequence of points $x^0,x^1,\ldots,x^N$ generated by Algorithm \ref{GeodesicAlgorithm} with stepsizes $t_k = \Delta := \frac{1}{\sqrt{2(N+1)}}$ satisfies
\[
\min_{k=0,\dots,N}F(x^k) - \min F ~\leq~ \frac{2\sqrt{2}D^2}{\sqrt{N+1}}.
\]
	\end{theorem}
	\begin{proof}
Since $x^0 \in A$, we first observe
\[
d^2(x^0,x_A) ~\leq~ F(x_A) ~\leq~ F(x^0) ~\leq~ D^2.
\]
At each step $k=0,1,\ldots,N-1$, if $F(x^k) \leq D^2$, then convexity of the squared distance function $d^2( \cdot,a)$ for all $a\in A$ 
	implies
	\[F(x^{k+1}) ~\leq~ (1-\Delta)F(x^k) + \Delta F(\bar a^k) ~\leq~ (1-\Delta)D^2 + \Delta D^2 ~=~ D^2.\]
By induction, we deduce $F(x^k) \leq D^2$ for all $k$.

		At step $k$, apply Lemma \ref{ProxDescent} with $x=x^k, y= \bar a^k, z= x_A$ and $\beta =\Delta$ to see
		\begin{align*}
			d^2(x^{k+1},x_A) &\le  \frac{(1-\Delta)d^2(x^k,x_A) + \Delta d^2(\bar a^k,x_A)}{1-\Delta}
			-\Delta(1-\Delta) d^2(x^k,\bar a^k) \\
			&= d^2(x^k,x_A) - \frac{\Delta}{1-\Delta} \left( (1-\Delta)^2d^2(x^k,\bar a^k) - 
			d^2(\bar a^k,x_A)\right)\\
			&\leq d^2(x^k,x_A) -\frac{\Delta}{1-\Delta}\left( \left(1-\Delta\right)^2 F(x^k) - \min F  \right),
		\end{align*}
		where the last inequality uses the property $d^2(\bar a^k,x_A) \leq F(x_A) = \min F$ and the
identity $F(x^k) = d^2(x^k,\bar a^k)$ follows from the definition of the point $\bar a^k$.
We deduce
	\[
		\label{sumineq}
		\frac{\Delta}{1-\Delta}\sum_{k=0}^N \left( \left(1-\Delta\right)^2F(x^k) - \min F \right)
		~\leq~ d^2(x^0,x_A) - d^2(x^{N+1},x_A) ~\leq~ D^2.
\]
Consequently we have
	\begin{align*}
		\frac{D^2}{\Delta}  \geq \frac{(1-\Delta)D^2}{\Delta} &\geq \sum_{k=0}^N \left( \left(1-\Delta\right)^2F(x^k) - \min F \right)\\
		& \geq 
		(N+1)\min_{k=0,\dots,N} \left( \left(1-\Delta\right)^2F(x^k) - \min F \right)\\
		&\geq (N+1) \left(-2\Delta D^2 + \min_{k=0,\dots,N} F(x^k) - \min F \right),
\end{align*}
using the nonnegativity of $F$ and the bound $F(x^k) \leq D^2$.
	The result now follows.
	\end{proof}

Strictly speaking, Theorems \ref{thm:FastComplexity} and \ref{thm:cat(k)complexity} concern the complexity of the surrogate problem of minimizing the objective $F$, but consequent complexity results for the original minimal enclosing ball problem are straightforward to deduce.  In the case of Hadamard space, suppose that Assumption \ref{assumptionA} holds, and consider any tolerance $\epsilon > 0$.  Then, running Algorithm \ref{GeodesicAlgorithm} with stepsizes $t_k = \frac{2}{k+1}$ for 
$N = \lfloor \frac{8\sqrt{\min F}}{\epsilon} \rfloor$ iterations ensures that one of the iterates $x^k$ is the center of an enclosing ball for $A$ whose radius exceeds the minimal enclosing ball radius by no more than $\epsilon$.  The same conclusion holds in CAT($\kappa$) spaces, now under Assumption \ref{assumptionB}, after  running Algorithm \ref{GeodesicAlgorithm} for $N = \lfloor\frac{32\min F}{\epsilon^2}\rfloor$ iterations with constant stepsize $t_k = \frac{1}{\sqrt{2(N+1)}}$.  Since the details of the argument are routine, we omit them.

	To illustrate Algorithm \ref{GeodesicAlgorithm} we construct a simple 
	but non-Riemannian CAT(1) space on the surface of the Euclidean sphere
	$S^2$. Consider two identical spherical caps centered at the north pole. Rotate each cap about the $x$-axis, one clockwise and the other counterclockwise, such that the
	two caps maintain a small overlap. Provided that the caps have sufficiently small radius, the union of the two caps
	endowed with the length metric induced from the angular metric on $S^2$ is a CAT(1) space satisfying Assumption
	\ref{assumptionB}(i)
	\cite[II.11.1]{Bridson1999}. This space and the outcome of running Algorithm \ref{GeodesicAlgorithm} on a
	finite set $A$ are visualized in Figure \ref{fig:experiment}.
	
	In this experiment the radius of each spherical cap is $\frac{5\pi}{36} - \delta$ and the offset from the north pole 
	is $\frac{\pi}{12}+\delta$, where $\delta = \frac{17}{200}$ and these quantities are measured in radians along
	the surface of the sphere. Figure \ref{fig:experiment} shows the space $X$ as the shaded portion of the sphere,
	the set $A$ consisting of 5 points, the trajectory of the algorithm's iterates when run with $t_k = (2(N+1))^{-1/2}$ 
	for $N=10^7$ iterations,
	and the approximate center and enclosing ball. The enclosing ball in $X$ is the intersection of the plotted ball
	with the shaded portion.

	\begin{center}
	\begin{figure}[h!]
		\hspace{0cm}\includegraphics[trim={0 10cm 0 4.5cm},clip, scale=0.7]{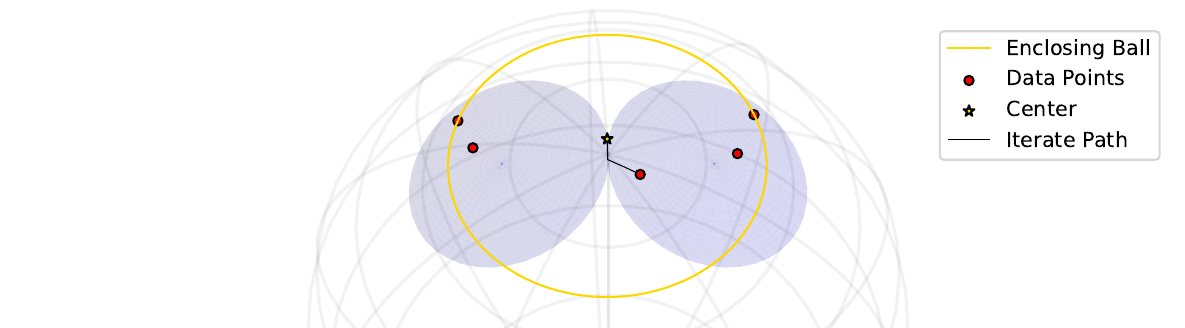}
		\caption{Computing the minimal enclosing ball on a CAT(1) space}
\label{fig:experiment}
	\end{figure}
	\end{center}
		
	\bibliographystyle{plain}
	\small
	\bibliography{ref}
\end{document}